\documentclass[11pt]{amsart}
\usepackage{etex} 

\usepackage[a4paper, margin=1.2in]{geometry} 

\usepackage{amsmath}
\usepackage{amssymb}
\usepackage{amsthm}
\usepackage{amsfonts}
\usepackage{booktabs}
\usepackage[usenames,dvipsnames,svgnames]{xcolor}
\usepackage{epsfig}
\usepackage{fancyvrb} 
\usepackage{hyperref}
\usepackage{mathrsfs}
\usepackage{mathtools}
\usepackage{pictexwd, dcpic}
\usepackage{pinlabel}
\usepackage{stmaryrd}
\usepackage{enumitem} 

\usepackage{tikz}
\usepackage{tikz-cd}
\usepackage{caption}
\usetikzlibrary{decorations.markings}
\usepackage{xypic}

\newcommand{\RR}{\mathbb{R}}

\newcommand{\ZZ}{\mathbb{Z}}

 \newcommand{\inte}{\operatorname{int}}

\newcommand{\vB}{\mathcal{B}}
\newcommand{\vC}{\mathcal{C}}

\newcommand{\vF}{\mathcal{F}}

\newcommand{\SL}{\operatorname{SL}}

\newcommand{\Aut}{\operatorname{Aut}}

\newcommand{\Ends}{\operatorname{Ends}}

\newcommand{\Homeo}{\operatorname{Homeo}}

\newcommand{\PMap}{\operatorname{PMap}}
\newcommand{\Map}{\operatorname{Map}}

\newcommand{\Out}{\operatorname{Out}}

\newcommand{\sm}{\setminus}
\newcommand{\ol}[1]{\overline{#1}}
\newcommand{\wt}[1]{\widetilde{#1}}

\newcommand{\id}{\text{id}}

\newcommand{\eand}{\quad \text{ and } \quad}

\definecolor{lightgrey}{gray}{.85}

\theoremstyle{definition}
\newtheorem*{defn}{Definition}
\newtheorem*{rmk}{Remark}

\newtheorem*{theorem}{Theorem}

\theoremstyle{plain}
\newtheorem{thm}{Theorem}[section]
\newtheorem{main}{Theorem}

\newtheorem{lem}[thm]{Lemma}
\newtheorem{cor}[thm]{Corollary}

\newtheorem{prop}[thm]{Proposition}

\newtheorem{qu}[thm]{Question}

\theoremstyle{definition}

\begin{document}
\title{Big mapping class groups and the co-Hopfian property}

\author{Javier Aramayona}
\address{Instituto de Ciencias Matem\'aticas, ICMAT (CSIC-UAM-UC3M-UCM)}
\email{javier.aramayona@icmat.es}
\thanks{J.A. was supported by grant PGC2018-101179-B-I00. He acknowledges financial support from the Spanish Ministry of Science and Innovation, through the “Severo Ochoa Programme for Centres of Excellence in R\&D” (CEX2019-000904-S). C.L. was partially supported by NSF grants DMS-1811518 and DMS-2106419.
A.M was supported by the Luxembourg National Research Fund OPEN grant O19/13865598.}

\author{Christopher J. Leininger}
\address{Rice University}
\email{c.j.leininger95@gmail.com}

\author{Alan McLeay}
\address{University of Luxembourg}
\email{mcleay.math@gmail.com}

\maketitle

\begin{abstract}
We study injective homomorphisms between big mapping class groups of infinite-type surfaces. First, we construct (uncountably many) examples of surfaces without boundary whose (pure) mapping class groups are not co-Hopfian; these are the first such examples of injective endomorphisms of mapping class groups that fail to be surjective. 

We then prove that, subject to some topological conditions on the domain surface, any continuous injective homomorphism between (arbitrary) big mapping class groups that sends Dehn twists to Dehn twists is induced by a subsurface embedding. 

Finally, we explore the extent to which, in stark contrast to the finite-type case, superinjective  maps between curve graphs impose no topological restrictions on the underlying surfaces.
\end{abstract}


\section{Introduction and main results} Throughout this article, all surfaces will be assumed to be connected, orientable and second-countable, unless otherwise specified. A surface $S$ has {\em finite type} if its fundamental group is finitely generated, and has {\em infinite type} otherwise.

 The mapping class group $\Map(S)$ is the group of orientation-preserving homeomorphisms of $S$, up to homotopy; if $\partial S\ne \emptyset$, then we require that all homeomorphisms and homotopies fix $\partial S$ pointwise. In the case when $S$ has finite type, $\Map(S)$ is well-known to be finitely presented. If, on the contrary, $S$ has infinite type, then $\Map(S)$ becomes an uncountable, totally disconnected, non-locally compact topological group with respect to the quotient topology stemming from the  compact-open topology on the homeomorphism group of $S$. We refer the reader to Section \ref{sec:prelim} for expanded definitions, and to the recent survey \cite{AV} for a detailed treatment of mapping class groups of infinite-type surfaces, now commonly known as {\em big mapping class groups}. 
 
\subsection{Algebraic rigidity for mapping class groups} A seminal result of Ivanov \cite{Ivanov} states that if $S$ is a (sufficiently complicated) finite type surface $S$, then every automorphism of $\Map(S)$ is a conjugation by an element of the {\em extended mapping class group} $\Map^\pm(S)$, namely the group of all homeomorphisms of $S$ up to isotopy. In other words, every isomorphism $\Map(S) \to \Map(S)$ is induced by a homeomorphism $S\to S$. The analog for infinite-type surfaces was recently obtained by Bavard-Dowdall-Rafi \cite{BDR}; see Theorem~\ref{thm:BDR} below.

The proofs in the finite and infinite type settings proceed along similar lines, based on the  work of Ivanov \cite{Ivanov}. First, one proves an algebraic characterization of Dehn twists and uses it to deduce that any isomorphism $\phi: \Map(S) \to \Map(S)$ must send (powers of) Dehn twists to (powers of) Dehn twists. In particular, $\phi$ induces a simplicial automorphism $\phi_*: \vC(S) \to \vC(S)$ of the {\em curve complex}. At this point, the argument boils down to showing that every simplicial automorphism of $\vC(S)$ is induced by a homeomorphism of $S$ \cite{Ivanov}. 
In fact, this approach applies, with a finite-list of exceptional surfaces, to any isomorphism between finite-index subgroups of $\Map(S)$ or the {\em pure mapping class group} $\PMap(S)$, the subgroup of $\Map(S)$ whose elements fix every end of $S$ \cite{Ivanov,BM,BelM,Irmak,Irmak2,Irmak3,Sha}

\subsection{The co-Hopfian property} The combination of results of Ivanov-McCarthy \cite{IM} and Bell-Margalit \cite{BelM} imply that, with the possible exception of the twice-punctured torus, if $S$ has finite type then every injective homomorphism $\Map(S) \to \Map(S)$  is induced by a homeomorphism of $S$. In particular, the mapping class group of a (sufficiently complicated) finite-type  surface is co-Hopfian: every injective endomorphism is surjective.

For infinite-type surfaces, Question 4.5 of the AIM Problem List on Surfaces of Infinite Type \cite{AIM} (see also \cite[Question 5.2]{AV}) asks: 

\begin{qu}
Are big mapping class groups co-Hopfian? 
\label{qu:main}
\end{qu}

%

\begin{rmk}
The answer to the above question is immediately ``no'' for surfaces with non-empty boundary. For instance, if a surface $S$ has non-empty boundary  and its space of ends is homeomorphic to  a Cantor set, then there exists a proper $\pi_1$-injective continuous map $S \to S$ that is not surjective. In turn, this map induces an injective homomorphism $\Map(S) \to \Map(S)$ that is not surjective; compare with Theorem \ref{thm:doubling} below. 
\end{rmk}

Our first result states that the answer to Question \ref{qu:main} is also negative for surfaces without boundary. Recall that the {\em Loch Ness Monster} is the unique (up to homeomorphism) connected orientable surface of infinite genus and exactly one end. We will prove:

\begin{main}
Let $S$ be either the once-punctured Loch Ness Monster or a torus minus the union of a Cantor set and an isolated point. Then there exists a homomorphism $\phi: \Map(S) \to \Map(S)$ such that: 
\begin{enumerate}
\item \label{item:main noncohopf} $\phi$ is continuous and injective, but not surjective,
\item \label{item:main nonfinitetype} There is a Dehn twist $t$ such that $\phi(t)$ is not supported on any finite-type subsurface of $S$. In the particular case when $S$ is the Loch Ness Monster,  no power of $\phi(t)$ is supported on a finite-type subsurface of $S$. 
\item \label{item:main nongeometric} There exists a partial pseudo-Anosov  $f\in\Map(S) $ such that $\phi(f)$ is a multitwist. 
\end{enumerate}
\label{thm:noncoHopf}
\end{main}

\begin{rmk}
Part (2) of the above theorem yields a negative answer to Problem 4.75 of the AIM Problem List on Surfaces of Infinite Type \cite{AIM}. 
\end{rmk}

The construction behind Theorem \ref{thm:noncoHopf} is  inspired by the construction of the {\em non-geometric} injective homomorphism of the first two authors with Souto \cite[Theorem 2]{ALS}, building on a well-known construction of homomorphisms from covers (see e.g. \cite{BirHil,IM,Win}). Namely, we construct a covering map $S' \to S$ which induces an injective homomorphism $\Map(S) \to \Map(S')$, in such a way that the surface $S''$ obtained by filling in all but one of the punctures of $S'$ is homeomorphic to $S$, and yet the homomorphism $\Map(S) \to \Map(S'')$ remains injective. Once this has been done, the resulting homomorphism is easily seen to not be surjective in light of Bavard-Dowdall-Rafi's result \cite{BDR} mentioned above, since $\phi$ sends some finitely supported elements to elements which are not finitely supported.

Motivated by this construction, we also observe the following.

\begin{thm}
Let $S$ denote the closed surface of genus $g\ge 1$ minus the union of a Cantor set and an isolated point.  Then there exists a continuous injective homomorphism $\Map(S) \to \Map(\RR^2 \smallsetminus C)$, where $C$ denotes a Cantor set. 
\label{thm:decreasing}
\end{thm}

As far as we know, these are the first examples of injective homomorphisms between mapping class groups where the genus decreases from domain to codomain. 

\subsection{Not co-Hopfian pure mapping class groups} If we restrict Question \ref{qu:main} to pure mapping class groups, we will see that there is a much richer palette of examples of injective, but not surjective, homomorphisms.
We say that a surface $S$ is \emph{self-doubling} if there exists a multicurve $\vB$ such that;
\begin{enumerate}
\item $S \sm \vB$ has two connected components $L$ and $R$,
\item  $L$ and $R$ are both homeomorphic to $S$, and
\item there exists orientation-reversing homeomorphism $\iota:S\to S$ of order two such that $\iota(L) = R$ and $\iota(b) = b$ for each $b \in \vB$.
\end{enumerate}

\begin{main}
For  every self-doubling surface $S$ there exists a continuous injective homomorphism $\PMap(S) \to \PMap(S)$ that is not surjective. 
\label{thm:doubling} 
\end{main}

%
%

\noindent{\bf Examples.} We now list some concrete examples of surfaces to which Theorem \ref{thm:doubling} applies. 
\begin{enumerate}
\item A first  example of a self-doubling surface $S$ is the sphere minus the union of a Cantor set and the north and south pole: any essential curve $b$ that separates the poles splits $S$ into two subsurfaces $L$ and $R$ with interiors homeomorphic to $S$.

\item Alternatively, suppose $S$ has infinite genus such that every planar end is isolated, and for every non-planar end there is a sequence of planar ends converging to it. (We remark that there are uncountably many surfaces with this property.) Then $S$ is also self-doubling; see Theorem \ref{thm:double-hard}. 
\end{enumerate}

We give an equivalent definition of self-doubling surfaces, along with more examples, in Section \ref{sec:doubling}.
%
%
%

\subsection{Twist-preserving homomorphisms} In light of the discussion above, we now focus on homomorphisms that send Dehn twists to Dehn twists; we call these {\em twist-preserving homomorphisms}. In this section we will allow surfaces to have a non-empty boundary, though we assume it is compact (and is thus a finite union of circles). We will prove the following result; recall that a map between topological spaces is {\em proper} if the preimage of every compact set is compact:

\begin{main}
Let $S$ and $S'$ be  surfaces of infinite type, where $S$ has positive genus. Assume further that either the boundary of $S$ is empty, or else $S$ has at most one end accummulated by genus.  If $\phi: \PMap(S) \to \PMap(S')$  is a continuous injective twist-preserving homomorphism,  then there is a proper $\pi_1$-injective embedding $ S \to S'$ that induces $\phi$. 
\label{thm:injections}
\end{main}

\begin{rmk} Theorem \ref{thm:injections} no longer holds if $\partial S \ne \emptyset$ and $S$ has more than one end accumulated by genus; see the final remark of Section \ref{sec:twist}. Also, in the same remark we will see that the result we will prove in fact applies to a larger class of homomorphisms than injective ones. 
\end{rmk}
%
%

Continuing with our discussion, observe that if $\partial S= \emptyset$, any proper $\pi_1$-injective embedding $S \to S'$ is homotopic to a homeomorphism. Hence we obtain: 

\begin{cor}
Let $S$ and $S'$ be  surfaces of infinite type, where $S$ has positive genus and no boundary. If $\phi: \PMap(S) \to \PMap(S')$ is a continuous injective twist-preserving homomorphism, then it is induced by a homeomorphism $S \to S'$; in particular,  it is surjective. 
\end{cor}


To prove Theorem \ref{thm:injections}, one first observes that $\phi$ induces a simplicial map $\vC(S) \to \vC(S')$ that preserves  intersection number one. Once this has been shown, the result follows considering exhaustions and a result of Souto and the first named author \cite{AS}, plus  a continuity argument.
We conjecture that, although continuity is needed in our proof, it is in fact not necessary for Theorem \ref{thm:injections} to hold. 
In this direction, a result of Mann \cite{Mann} states that the mapping class group of certain surfaces have the {\em automatic continuity} property. Specifically, this automatic continuity holds when the surface is a closed surface punctured along the union of a Cantor set and a finite set; for these surfaces we obtain: 

\begin{cor} \label{C:injections+Mann}
Let $S$ and $S'$ be  surfaces of infinite type, where $S$ is obtained from a closed surface of positive genus by removing the union of a Cantor set and a finite set. If $\phi: \PMap(S) \to \PMap(S')$  is an injective twist-preserving homomorphism,  then there is a  homeomorphism $S \to S'$ that induces $\phi$. 
\end{cor}

\subsection{Injections between curve graphs}

In the finite-type setting, an important step for establishing the co-Hopfian property for finite-index subgroups of mapping class groups is that {\em superinjective} self-maps of the curve graph are induced by homeomorphisms; see \cite{Irmak}, for instance. We recall that  a simplicial map between curve graphs is {\em superinjective} if it maps pairs of curves with non-zero intersection number to pairs of curves with the same property. 

 While it is known that every automorphism of the curve graph of an infinite-type surface is induced by a homeomorphism, \cite{HMV-iso,BDR}, it is easy to see that this is no longer true if one drops the requirement that the map be surjective; see \cite{HV,AV} for concrete examples. Our final result highlights the extent of this failure:

\begin{main}\label{thm:superinjective}
Let $S$ and $S'$ be infinite-type surfaces, where $S'$ has infinite genus. Then: 
\begin{enumerate}
\item There exists a superinjective simplicial map $\vC(S) \to \vC(S)$ that is not surjective. 
\item There exists a superinjective simplicial map $\vC(S) \to \vC(S')$.
\end{enumerate}
\end{main}

In the case of infinite-genus surfaces, a proof of part (1) Theorem \ref{thm:superinjective} was previously obtained by Hern\'andez-Valdez \cite{HV} using different methods.

\begin{rmk}
If $S$ and $S'$ are arbitrary infinite-type surfaces, then there is always an {\em injective}  simplicial map $\vC(S) \to \vC(S')$. To see this, simply observe that the curve graph of any infinite-type surface contains the complete graph on countably many vertices, and that the curve graph of any surface has a countable set of vertices. 
\end{rmk}

Theorem \ref{thm:superinjective} should be compared with Theorem \ref{thm:injections}, which implies that every superinjective map between curve graphs {\em that comes from an injective homomorphism between the corresponding mapping class groups} is induced by an embedding of the underlying surfaces. 

\medskip

\noindent{\bf Plan of the paper.} In Section 2 we will give all the necessary background needed in the article. Sections \ref{sec:torus} and \ref{sec:monster} are devoted to the proof of Theorem \ref{thm:noncoHopf} in the cases, respectively, of the one-holed torus minus a Cantor set and the once-punctured Loch Ness Monster. In Section \ref{sec:doubling} we deal with Theorem \ref{thm:doubling}.  Section \ref{sec:twist} is devoted to the proof Theorem \ref{thm:injections}. Finally, in Section \ref{sec:graphs} we establish Theorem \ref{thm:superinjective}. 

\medskip

\noindent{\bf Acknowledgements.} We are grateful to the organizers of the AIM workshop ``Surfaces of Infinite Type'' for the discussions that are the origins of this paper. Thanks also to AIM for its hospitality and financial support. We would like to thank Nick Vlamis and Henry Wilton for enlightening conversations about Lemma \ref{lem:monster}, and Vlamis for the reference \cite{Dom}. We also thank Israel Morales for bringing to our attention the refererence \cite{ARH}, and Ty Ghaswala for comments on an earlier version. 

\section{Preliminaries}\label{sec:prelim}

In this section we will briefly introduce all the background material needed for our results. We refer the reader to the survey paper \cite{AV} for a detailed account on these topics. 

\subsection{Surfaces} In what follows, by a {\em surface} we will mean a  orientable second-countable topological surface, with a finite number (possibly zero) of boundary components, all of them assumed to be compact. If the fundamental group of a surface $S$ is finitely generated, we will say that $S$ has {\em finite type}; otherwise it has {\em infinite type}. The {\em space of ends} of $S$ is \[\Ends(S):=\lim_{\leftarrow} \pi_0(S\smallsetminus K),\] where $K$ denotes a compact subset of $S$. The space of ends becomes a topological space with respect to  the final topology obtained from endowing each set defining the inverse system with the discrete topology. It is well-known that $\Ends(S)$ is a closed subset of a Cantor set.  We say that an end is {\em planar} if it has a planar neighborhood; otherwise we will say that the end is {\em non-planar} or that it is {\em accumulated by genus}. An isolated planar end is usually called a {\em puncture}; as is customary, we will treat punctures both as ends and as marked points on the surface, switching between the two viewpoints without further mention. 
Finally, we will denote by $\Ends_\infty(S)$ the subspace of non-planar ends, which is closed in $\Ends(S)$.

A classical result \cite{Richards} asserts that any surface $S$  is uniquely determined by the quadruple $(g,b, \Ends(S), \Ends_\infty(S))$, where $g\in \mathbb{N} \cup \{\infty\}$ is the genus, $b\in \mathbb{N}$ is the number of boundary components. More concretely, two surfaces are homeomorphic if and only if their genera and number of boundary components  are equal, and there is a homeomorphism between the spaces of ends that restricts to a homeomorphism between the spaces of non-planar ends. 


\subsection{Curves and multicurves} A simple closed curve on $S$ is said to be {\em inessential} if it bounds a disk, a once-punctured disk, or an annulus whose other boundary is a boundary component of $S$; otherwise we say the curve is {\em essential}.  By a {\em curve} we mean the isotopy class of an essential simple closed curve on $S$. A {\em multicurve} is a set of curves on the surface that have pairwise disjoint representatives. Given a multicurve $Q$, we will write $S \smallsetminus  Q$ to mean the (possibly disconnected) surface obtained from $S$ by cutting along (pairwise disjoint representatives of) each element of $Q$.

%

\subsection{Mapping class group} Let $\Homeo^+(S)$ be the group of all orientation-preserving self-homeomorphisms of $S$, equipped with the compact-open topology. We will denote by $\Homeo_0(S)$ the connected component of the identity in $\Homeo^+(S)$, noting it is a normal subgroup. The {\em mapping class group} of $S$ is \[\Map(S):=\Homeo^+(S) / \Homeo_0(S).\]
As is customary, if $\partial S \ne \emptyset$, in this definition we implicitly require that all homeomorphisms fix $\partial S$ pointwise.

The mapping class group is naturally a topological group with respect to the quotient topology coming from the compact-open topology on $\Homeo^+(S)$. It is a classical fact that $\Map(S)$ is a finitely presented group if $S$ has finite type, while it is easy to see that it is uncountable whenever $S$ has infinite type. Moreover, in this latter case, $\Map(S)$ is totally disconnected and not locally compact; see \cite{AV} for more details.

One of the motivating results for us is the following theorem of Bavard-Dowdall-Rafi \cite{BDR} mentioned in the introduction.

\begin{theorem}[\cite{BDR}]
For any infinite-type surface $S$, any isomorphism $\Map(S) \to \Map(S)$ is induced by a homeomorphism $S\to S$. In particular, any such isomorphism is continuous. 
\label{thm:BDR}
\end{theorem}

The {\em pure mapping class group} $\PMap(S)$ is the subgroup of $\Map(S)$ whose elements fix every end of $S$.  If $S$ has finite type, it is well-known that $\PMap(S)$ is generated by Dehn twists along a finite set of curves (see \cite[Section 4.4]{FM}). When $S$ has infinite type, this is no longer true in general; instead, we have the following result; see \cite{PV} for the definition of a {\em handle shift}: 

\begin{thm}[\cite{PV}]
$\PMap(S)$ is topologically generated by Dehn twists if and only if $S$ has at most one end accumulated by genus. Otherwise, it is topologically generated by Dehn twists and  handle shifts. 
\end{thm}
Here, being {\em topologically generated} by Dehn twists means that the subgroup generated by Dehn twists (i.e. the {\em compactly-supported} mapping class group, see below) is dense in $\PMap(S)$. 

\subsubsection{Some important subgroups}  Following \cite{BDR}, we will say that an element of $\Map(S)$ has {\em finite support} if it represented by a homeomorphism which is the identity outside some finite-type subsurface of $S$. We will write $\Map_f(S)$ for the subgroup consisting of elements with finite support. Although not directly relevant to our arguments, we record the following beautiful result of Bavard-Dowdall-Rafi \cite{BDR}, which gives an algebraic characterization of elements with finite support. In particular, it serves to shed light on the potential  differences between isomorphisms and injective homomorphisms between big mapping class groups.

\begin{prop}[\cite{BDR}]
\label{prop:charcompact}
An element of $\Map(S)$ has finite support if and only if its conjugacy class in $\Map(S)$ is countable. 
\end{prop}

The {\em compactly-supported} mapping class group $\Map_c(S)$ is the subgroup of $\Map(S)$ whose elements are represented by homeomorphisms which are the identity outside some compact subsurface of $S$. Observe that $\Map_c(S)$ is a subgroup of $\Map_f(S)$, proper when $S$ has more than one puncture, and of $\PMap(S)$.  On the other hand, $\Map_f(S)$ is a subgroup of $\PMap(S)$ if and only if $S$ has at most one puncture. Before ending this section, we record the following immediate observation for future use: 

\begin{lem}
\label{lem:direct}
For every infinite-type surface $S$,  we have \[\Map_c(S) = \lim_{\longrightarrow} \Map(X) ,\] where the direct limit is taken over all compact subsurfaces $X \subset S$, directed with respect to inclusion. 
\end{lem}

\subsection{Curve graph}
The {\em curve graph} $\mathcal{C}(S)$ is the simplicial graph whose vertex set is the set of curves on $S$, and where two vertices are deemed to be adjacent if the corresponding curves may be realized disjointly on $S$. In what follows we will not distinguish between vertices of the curve graph and the curves representing them. 

Observe that $\Map(S)$ acts on $\mathcal{C}(S)$ by simplicial automorphisms. A classical fact, discovered initially by Ivanov \cite{Ivanov} is that every simplicial automorphism of $\vC(S)$ is induced by a homeomorphism of $S$, provided $S$ is not the twice-punctured torus. The analog for infinite-type surfaces has been obtained independently by Hern\'andez-Morales-Valdez \cite{HMV} and Bavard-Dowdall-Rafi \cite{BDR}.
A crucial step in their proofs is the following so-called {\em Alexander method} for infinite-type surfaces \cite{HMV}, which we record for future use: 

\begin{thm} \label{T:Alexander}
Let $S$ be a connected orientable infinite-type surface with empty boundary. The natural action of $\Map(S)$ on $\mathcal{C}(S)$ has trivial kernel; in other words, if $f\in \Map(S)$ induces the identity transformation on $\mathcal{C}(S)$, then it is the identity in $\Map(S)$. 
\end{thm}

\section{Covers and forgetting}

In this section we prove Theorem~\ref{thm:noncoHopf} and Theorem \ref{thm:decreasing}.  These theorems are about homomorphisms between mapping class groups of surfaces with a single puncture, and the proofs exploit their relationship with automorphism groups and the Birman exact sequence.  We start with some generalities, then focus on the specific cases of the theorems.

\subsection{Covers and automorphism groups} In what follows, suppose $S$ is an orientable, connected surface without punctures such that $\pi_1(S)$ is non-abelian.
Given $z \in S$ any point, let $S^z  = S \smallsetminus \{z\}$ denote the surface obtained from $S$ by removing $z$ (thus producing a puncture we call the {\em $z$--puncture}).  Any homeomorphism $f$ of $S^z $ induces a homeomorphism of $S$ that fixes the point $z$, which by an abuse of notation we also denote by $f$.  This defines a homomorphism
\[ \Homeo^+(S^z ) \to \Homeo^+(S). \]
This homomorphism sends $\Homeo_0(S^z )$ into $\Homeo_0(S)$ and so induces a homomorphism
\[ \Map(S^z ) \to \Map(S).\]

Given a loop $\gamma$ representing an element of $\pi_1(S,z)$, one associates a homeomorphism $f_\gamma \colon S^z  \to S^z $ by ``point pushing along $\gamma$".  More precisely, $f_\gamma \colon S \to S$ is a homeomorphism isotopic to the identity by an isotopy $f_t$ with $f_0 = f_\gamma$, $f_1 = \id$, and $f_t(z) = \gamma(t)$, for all $t \in [0,1]$.  When $S$ has finite type, Birman proved that $[\gamma] \mapsto [f_\gamma]$ defines an injective homomorphism $\pi_1(S,z) \to \Map(S^z )$ in such a way that the sequence
\begin{equation} \label{E:BES} \xymatrix{ 1 \ar[r] & \pi_1(S,z) \ar[r] & \Map(S^z ) \ar[r] & \Map(S) \ar[r] & 1 } \end{equation}
is exact; see \cite{Bir}.  The infinite type case is proved by Dickmann-Domat in the appendix to Domat's paper \cite{Dom}.
\begin{prop} \label{P:BES}
For any connected, orientable surface $S$ without punctures and non-abelian fundamental group, there is an injection of $\pi_1(S,z) \to \Map(S^z )$ given by $[\gamma] \mapsto [f_\gamma]$ making \eqref{E:BES} exact. \qed
\end{prop}
For the remainder of this section, we use this theorem to identify $\pi_1(S,z)$ with its image in $\Map(S^z)$.

\medskip

\noindent {\bf Remark.} The theorem also holds when $S$ has punctures, replacing $\Map(S^z )$ with the subgroup preserving the $z$--puncture.

\medskip

Given $f \in \Homeo^+(S^z )$, after extending over $z$ we have an induced automorphism $f_* \in \Aut(\pi_1(S,z))$, which descends to a homomorphism $\iota \colon \Map(S^z ) \to \Aut(\pi_1(S,z))$ given by $\iota([f]) = f_*$.  This further descends to a homomorphism $\Map(S) \to \Out(\pi_1(S))$ which is injective by Theorem~\ref{T:Alexander}.  The inclusion of $\pi_1(S,z)$ into $\Map(S^z )$, composed with $\iota$ is precisely the isomorphism onto the group of inner automorphisms, and thus we get a homomorphism of short exact sequences:
\[ \xymatrixrowsep{.5cm}\xymatrix{ 1 \ar[r] & \pi_1(S,z) \ar[r] \ar[d] & \Map(S^z ) \ar[r] \ar[d]^\iota & \Map(S) \ar[r] \ar[d] & 1 \,\, \\
1 \ar[r] & \pi_1(S,z) \ar[r] & \Aut(\pi_1(S,z)) \ar[r] & \Out(\pi_1(S,z)) \ar[r] & 1.
}\]
The first vertical map is the identity and the last is injective, from which it follows that the middle is also injective.

A regular covering space $p \colon \wt S \to S$ with the property that every homeomorphism of $S$ lifts to a homeomorphism of $\wt S$ will be called a {\em geometrically characteristic} cover.  If $p_*(\pi_1(\wt S, \wt z))$ is characteristic, general map lifting implies that the cover is geometrically characteristic, but this is not a necessary condition in general; see the examples below.

Now suppose $p \colon \wt S \to S$ is geometrically characteristic, let $z \in S$ be a point, and fix any $\wt z \in p^{-1}(z)$.  Being regular and geometrically characteristic implies that for any $f \in \Homeo^+(S^z)$, after extending over $z$, there is a unique lift $\wt f \colon \wt S \to \wt S$ that fixes $\wt z$, and thus determines a homeomorphism of the same name $\wt f \in \Homeo^+(\wt S^{\wt z})$.  The general fact we need is the following (compare with \cite[Theorem~2]{ALS}).
\begin{prop} \label{P:lifting general} If $p \colon \wt S \to S$ is a geometrically characteristic covering space, $\pi_1(S)$ is non-abelian, and $\pi_1(\wt S)$ is non-trivial, then the assignment $f \mapsto \wt f$ described above descends to a continuous, injective homomorphism
\[ \phi \colon \Map(S^z) \to \Map(\wt S^{\wt z}).\]
Moreover, via the inclusions from the Birman exact sequence we have $\phi \circ p_* = \id|_{\pi_1(\wt S,\wt z)}$.
\end{prop}
\begin{proof}  Continuity of $\Homeo^+(S^z) \to \Homeo^+(\wt S^{\wt z})$ is clear.  Since $f \mapsto \wt f$  maps $\Homeo_0(S^z)$  to $\Homeo_0(\wt S,\wt z)$ (by lifting isotopies), we get a well-defined, continuous homomorphism $\phi \colon \Map(S^z) \to \Map(\wt S^{\wt z})$. Since the homomorphisms $\iota \colon \Map(S^z) \to \Aut(\pi_1(S,z))$ and $\tilde \iota \colon \Map(\wt S^{\wt z}) \to \Aut(\pi_1(\wt S,\wt z))$ are injective, to prove the first part of the proposition it suffices to prove injectivity of the homomorphism $\phi_* \colon \iota(\Map(S^z)) \to \wt \iota(\Map(\wt S^{\wt z}))$ defined by $\phi_* \circ \iota = \wt \iota \circ \phi$, or more explicitly
\[ \iota([f]) = f_* \stackrel{\phi_*}\longmapsto \tilde f_* =\tilde \iota([\tilde f]) = \tilde \iota(\phi([f])).\]
Consequently, we need only show that the kernel of this homomorphism is trivial.

To this end, suppose that $f \in \Homeo^+(S^z)$ is any element such that $\tilde f_*$ is the identity.
Let \[G = p_*(\pi_1(\wt S,\wt z)) \triangleleft \pi_1(S,z),\] so that $p_*$ is an isomorphism from $\pi_1(\wt S,\wt z)$ onto the normal subgroup $G$ (normality follows from regularity of $p$).  Since $\widetilde f$ is a lift of $f$ that preserves $\wt z$, we have that
\[ p_* \circ \tilde f_* = f_*  \circ  p_*.\]
In particular, this means that $f_*$ preserves $G$ and $p_*$ conjugates $\widetilde f_*$ to $f_*|_G$.  Since $\tilde f_*$ is the identity, this implies that $f_*|_G$ is the identity.  We claim that $f_*$ is the identity.  To prove this, first observe that for all $x \in G$ and $a \in \pi_1(S,z)$, we have $axa^{-1} \in G$, and thus 
\[ axa^{-1} = f_*(axa^{-1}) = f_*(a)f_*(x) f_*(a)^{-1} = f_*(a)xf_*(a)^{-1}.\]
This implies that for all $a \in \pi_1(S,z)$ and all $x \in G$, $a^{-1}f_*(a)$ commutes with $x$.

Since $\pi_1(S,z)$ is a non-abelian surface group (either free or a closed surface group), we know that centralizers of elements are cyclic, and since $G$ is a non-trivial normal subgroup of $\pi_1(S,z)$, we can find two elements $x,y \in G$ who centralizers intersect trivially.   Now for any $a \in \pi_1(S,z)$, $a^{-1}f_*(a)$ is in the centralizer of $x$ and $y$, and thus $a^{-1}f_*(a) = 1$.  Therefore, $f_*(a) = a$, and hence $f_*$ is the identity proving that $\phi_*$, and hence $\phi$, is injective.

To prove the last statement, given any element $g$ of a group, let $c_g$ denote the inner automorphism determined by conjugating by $g$.  From the inclusions in the Birman exact sequence as noted above, given any $[\gamma] \in \pi_1(\wt S,\wt z)$ we have
\[ \wt \iota([\gamma]) = c_{[\gamma]} \quad \mbox{ and } \quad \iota(p_*([\gamma])) = c_{p_*([\gamma])}.\]
Since $\phi_* \circ \iota = \wt \iota \circ \phi$, we have
\[ \wt \iota (\phi(p_*([\gamma]))) =  \phi_*(\iota(p_*([\gamma]))) = \phi_*(c_{p_*([\gamma])}) = c_{[\gamma]} = \wt \iota([\gamma]),\]
where the second-to-last equality comes from the fact that the isomorphism described above, $p_* \colon \pi_1(\wt S,\wt z) \to G< \pi_1(S,z)$, conjugates $c_{[\gamma]}$ to $c_{p_*([\gamma])}$, while the argument above shows that it conjugates $\phi_*(c_{p_*([\gamma])})$ to this as well, hence $c_{[\gamma]} = \phi_*(c_{p_*([\gamma])})$.
Therefore, $\wt \iota \circ \phi \circ p_* = \wt \iota$ and since $\wt \iota$ is injective, it follows that $\phi \circ p_*$ is the identity on $\pi_1(\wt S,\wt z)$, as required.
\end{proof}

\subsection{The once-punctured torus minus a Cantor set}
\label{sec:torus}

Here we prove Theorem \ref{thm:noncoHopf} in one of the two cases.
\begin{proof}[Proof of Theorem~\ref{thm:noncoHopf} for the once-punctured torus minus a Cantor set] 
Let $T$ be a torus and choose, once and for all, a Cantor set $C\subset T$, set $S = T \smallsetminus C$, fix $z \in S$, and write $S^z  = S \smallsetminus z$ as above.
Then $C \cup \{z\}$ is canonically homeomorphic to $\Ends(S^z )$ and $T$ is the end-compactification of $S^z $.  Consider the surjective homomorphism 
\[  \pi_1(S,z) \to \pi_1(T, z) \cong  \ZZ \times \ZZ\] given by ``filling in'' every element of $C$.
Fixing any integer $m \geq 2$ and reducing modulo $m$, we get a surjective homomorphism
\[
\rho:\pi_1(T,z) \to \ZZ_m \times \ZZ_m.
\]
Let $p:\wt T \to T$ be the regular cover corresponding to $\ker(\rho)$, and fix some $\tilde z \in p^{-1}(z)$.  Write $\wt C = p^{-1}(C)$, and observe that the surface $\wt S = \wt T \smallsetminus \wt C$ is homeomorphic to $S$, and $\wt S^{\wt z} = \wt S \smallsetminus \{\wt z\}$ is homeomorphic to $S^z$.

Every homeomorphism of $T$ lifts to $\wt T$ since the kernel of \[\pi_1(T,z) \to \ZZ_m \times \ZZ_m\] is characteristic.  Moreover, every homeomorphism of $S$ extends uniquely to a homeomorphism of $T$ and hence the restricted covering (with the same name) $p \colon \wt S \to S$ is geometrically characteristic.  Proposition~\ref{P:lifting general} therefore provides an injective, continuous homomorphism $\phi \colon \Map(S^z) \to \Map(\wt S^{\wt z})$.  Since $\wt S^{\wt z}$ and $S^z$ are homeomorphic surfaces,  by choosing a homeomorphism between them we can view $\phi$ as a continuous homomorphism $\phi \colon \Map(S^z) \to \Map(S^z)$.  We have already shown that $\phi$ is injective. We now verify the remaining properties claimed in the statement: \\

\noindent{\em Compact to non-finite support.} 
To prove part \eqref{item:main nonfinitetype} of Theorem \ref{thm:noncoHopf} we will show that the $\phi$--image of the Dehn twist $t_a \in \Map_c(S^z)$ about a non-separating curve $a$ is not supported on any finite-type subsurface of $S^z$. 

In order to achieve this it suffices to find a point of $\tilde C$ that is not fixed by our chosen lift $\wt t_a= \phi(t_a)$.  Let $\beta \subset S^z$ be an arc which begins at $z$ and ends at a point $e \in C$ so that $\beta$ essentially intersects $a$ exactly once.  If $\tilde \beta$ is the lift of $\beta$ starting at $\wt z$ then $\wt t_a(\tilde \beta)$ also starts at $\wt z$, but these two arcs terminate at distinct points in the preimage of $e$.  Therefore $\phi(t_a)$ does not have finite support. (Observe, however, that there is a nontrivial power of $\phi(t_a)$ that has finite support; this will not happen in the case of the Loch Ness Monster.) \\ 

\noindent{\em Not co-Hopfian.}  Suppose $\phi$ were surjective, and thus an automorphism.  Theorem~\ref{thm:BDR} implies that $\phi$ is induced by a homeomorphism of $S^z$, and in particular preserves the property of having finite support, which is a contradiction to the previous paragraph.  This proves part \eqref{item:main noncohopf} \\

\noindent {\em Non-geometric.}
To prove part \eqref{item:main nongeometric}, and thus complete the proof of Theoreom \ref{thm:noncoHopf}, we must show that there are partial pseudo-Anosovs in $\Map(S^z)$ that map to multitwists in $\Map(\wt S^{\wt z})$.  The proof is essentially the same as that of  \cite[Theorem 2(1)]{ALS}; we sketch it for completeness.

As in \cite{ALS}, one can find a loop $\gamma$ representing an element in $\pi_1(\wt S,\wt z)$ which is simple, but for which $p_*([\gamma])$ is not represented by any simple closed curve.  The mapping class associated to $[\gamma]$ via the Birman exact sequence is a multi-twist about the boundary of a regular neighborhood of $\gamma$, while by a result of Kra \cite{Kra}, $p_*([\gamma])$ is a pseudo-Anosov on $X \smallsetminus \{z\}$, where $X$ is the subsurface filled by a loop representing $p_*([\gamma])$ with minimal self-intersection.  Since $\phi \circ p_*([\gamma])= [\gamma]$ by Proposition~\ref{P:lifting general}, it follows that $p_*([\gamma])$ is pseudo-Anosov on a proper subsurface while $\phi(p_*([\gamma]))$ is a multi-twist.  This completes the proof of the theorem in this case.
\end{proof}

\subsection{The Loch Ness Monster}
\label{sec:monster}
In this section we prove Theorem \ref{thm:noncoHopf} for the once-punctured Loch Ness Monster. In what follows,  $S$ will denote the {\em Loch Ness Monster}, that is, the connected orientable surface of infinite genus and exactly one end. As in the case of a torus minus a Cantor set, we fix once and for all, a point $z\in S$.  We will again apply Proposition~\ref{P:lifting general} for an appropriate cover $p \colon \wt S \to S$ to induces an injective homomorphism $\phi \colon \Map(S^z) \to \Map(\wt S^{\wt z})$.  This time the cover will be of infinite degree which, in turn, will be the key to the existence of compactly-supported elements with image for which no non-trivial power has compact support. 
Consider the homomorphism \[\rho: \pi_1(S,z) \to H_1(S, \mathbb{Z}_2)\] obtained by first abelianizing and then reducing modulo 2. Observe that $H:=\ker(m)$ is a characteristic subgroup of $\pi_1(S,z)$. Let $p:\tilde S \to S$ be the cover associated to $H$, which is usually called the {\em mod-2 homology cover} of $S$. We claim: 

\begin{prop}
$\wt S$ is homeomorphic to $S$. 
\end{prop}

\begin{proof}
First, observe that the Loch Ness Monster $S$ is a characteristic cover of the closed surface $\Sigma$ of genus $2$; more precisely, it is the covering surface associated to the commutator subgroup of $\pi_1(\Sigma)$. Since $\wt S$ is a characteristic cover of $S$, it follows that $\wt S$ is also a characteristic cover of $\Sigma$. 

Let $H$ be the characteristic subgroup of $\pi_1(\Sigma)$ corresponding to $\pi_1(\wt S)$ and $D = \pi_1(\Sigma)/H$ the group of deck transformations of $\tilde S \to \Sigma$.  Since this action is properly discontinuous and cocompact, the \u{S}varc-Milnor Lemma (see e.g. \cite{BH}) implies that $D$ is quasi-isometric to $\wt S$.  By Stallings' Theorem \cite{Stallings} $D$, or equivalently $\tilde S$, has one end, two ends, or infinitely many ends.  By the classification of infinite-type surfaces \cite{Richards}, it follows that $\wt S$ is homeomorphism to either the Loch-Ness Monster $S$; Jacob's Ladder $L$; the Cantor tree $T$; or the Blooming Cantor Tree $B$; see \cite{BV} for details.   Since $S$ is the only one of these that has one end, we suppose $\wt S$ has more than one end and derive a contradiction.

To this end, we appeal to Stallings Theorem again and note that $D$ admits a non-trivial action on a tree $A$ with finite edge stabilizers and without edge inversions.
From this action construct an equivariant map $f \colon \wt S \to A$, which we can assume is transverse to the union of midpoints $X \subset T$ of every edge, so that $f^{-1}(X)$ is a properly embedded $1$--submanifold which is $D$--invariant (c.f. \cite{Shalen}). This descends to a closed $1$--submanifold in $\Sigma$, and via the action of $\pi_1(\Sigma) \to D$ on $A$, any loop in the complement of the $1$--submanifold fixes a vertex.  Observe that there is a non-separating simple closed curve in a component of this complement (this is true for any $1$--submanifold in the genus $2$ surface $\Sigma$) which, when viewed as an element of $\pi_1(\Sigma)$,  fixes a vertex. Since the cover is characteristic, every non-separating simple closed curve fixes a vertex of $A$.

Now choose a generating set $a_1,\ldots,a_4$ so that all $a_i$ are simple as are $a_ia_j$, for all $i \neq j$.  Then all $a_i$ and $a_ia_j$ act elliptically, in which case Serre's criterion \cite{Serre} implies $\pi_1(\Sigma)$ fixes a point of $A$, which is a contradiction.
\end{proof}

\begin{proof}[Proof of Theorem \ref{thm:noncoHopf} for the once-punctured Loch Ness Monster]
Fixing $\wt z \in p^{-1}(z)$, Proposition~\ref{P:lifting general} implies that $p$ induces an injective homomorphism $\phi \colon \Map(S^z) \to \Map(\wt S^{\wt z})$.

Since $p$ has infinite degree, there are Dehn twists whose image does not have compact support, even up to taking powers. For the same reason, $\phi$ is not surjective. The fact that there exist partial pseudo-Anosovs whose image is a Dehn twist follows along the exact same lines as in the torus case. Finally, since $\wt S$ and $S$ are homeomorphic, so are $\wt S^{\wt z}$ and $S^z$, and thus we may view the homomorphism $\phi$ as an injective, but not surjective, endomorphism of $\Map(S^z)$. This finishes the proof of the Theorem in the case of the Loch Ness Monster. 
\end{proof}

\begin{rmk}
After this article was finished, we learned that Proposition \ref{Monster} had already been established in \cite[Proposition 4.1]{ARH}.
\end{rmk}

\subsection{Decreasing genus}   Our final application of Proposition~\ref{P:lifting general} is the following.
\begin{proof}[Proof of Theorem~\ref{thm:decreasing}]
For $g\ge 1$, let $\Sigma_g$ denote the surface of genus $g\geq 1$, let $C_g \subset \Sigma_g$ be a Cantor set, and $S = \Sigma_g \smallsetminus C_g$.  Let $p \colon \wt \Sigma \to \Sigma_g$ be the universal cover.  Choosing a disk $D \subset \Sigma$ that contains $C_g$, we note that $p^{-1}(C_g)$ is a disjoint union of Cantor sets that only accumulate at infinity of $\wt \Sigma$.  Consequently, since the one-point compactification of $\wt \Sigma$ is the sphere, it follows that
\[ \wt S = \wt \Sigma \smallsetminus p^{-1}(C_g)\]
is homeomorphic to the $2$--sphere minus a Cantor set.

Now observe that $p \colon \wt S \to S$ is a geometrically characteristic cover, and so for any basepoint $z \in S$ and choice of $\wt z \in p^{-1}(z)$, Proposition~\ref{P:lifting general} implies that $p$ induces an injective homomorphism $\phi \colon \Map(S^z) \to \Map(\wt S^{\wt z})$.  Observe that $S^z$ is a surface of genus $g$ minus a Cantor set and an isolated point, while $\wt S^{\wt z}$ is a $2$--sphere minus a Cantor set and an isolated point, which is also homeomorphic to the plane $\mathbb R^2$ minus a Cantor set.  This completes the proof.
\end{proof}

\section{Self-Doubling}
\label{sec:doubling}
In this section we will prove Theorem \ref{thm:doubling} and provide a large class of self-doubling surfaces.  We will use an argument similar to the example in \cite[Section 2]{IM}, although adapted to the case of surfaces without boundary. 

\subsection{The construction} Let $\bar S$ be an orientable connected surface with non-empty boundary. In this section, we still require each boundary component of $\bar S$ be compact, although this time we allow the set of boundary components to be countable. Let $\vB = \{b_1, b_2, \dots \}$ be a (possibly finite) subset of the set of boundary components, and $S$ be the surface that results from $\bar S$ by gluing disks with marked points $z_i$ onto $b_i$. This operation gives rise to a {\em boundary deleting} homomorphism $ \PMap(\bar S) \to \PMap(S) $, which fits in a short exact sequence

\begin{equation}
\label{eq:delete}
1 \to \Pi_i \langle T_{b_i} \rangle \to \PMap(\ol S) \to \PMap(S) \to 1,
\end{equation}
see \cite{FM} for details. Now suppose $d_\vB S$ is the surface obtained by gluing two disjoint copies of $\bar S$ along the boundary components in $\vB$. This operation induces two natural inclusion maps
\[
\psi_1: \bar S \hookrightarrow d_\vB S \eand \psi_2: \bar S \hookrightarrow d_\vB S
\]
such that $\inte(\psi_1(\bar S))\cap \inte (\psi_2(\bar S)) =\emptyset$ and  $\psi_1(b_i) = \psi_2(b_i)$ is an essential curve. We abuse notation by writing $b_i$ and $\vB$ for the images of $b_i$ and $\vB$, respectively.  Furthermore, there is an induced orientation-reversing homeomorphism \[\iota: d_\vB S \to d_\vB S\] that swaps the images of $\psi_1$ and $\psi_2$ and, in particular, fixes the set $\vB \subset d_\vB S$.
On the level of mapping class groups, we have two injective homomorphisms
\[
\Psi_1 : \PMap(\bar S) \hookrightarrow \PMap(d_\vB S) \eand \Psi_2 : \PMap(\bar S) \hookrightarrow \PMap(d_\vB S),
\]
such that for all $f \in \PMap(\ol S)$, we have $\Psi_1(f) = f_1$, $\Psi_2(f) = f_2$, where $f_1 = \iota f_2 \iota^{-1}$; here we consider $f_1$, $f_2$, and $\iota$ as elements of $\Map^{\pm}(d_\vB S$).  In particular, a left twist about $b_i$ is sent to a left twist by $\Psi_1$, and to a right twist by $\Psi_2$.
Consider now the map
\[
d : \PMap(\bar S) \to \PMap (d_\vB S)
\]
such that $d(f) = f_1 f_2$.  Note that $d$ is indeed a homomorphism as the images of $\Psi_1$ and $\Psi_2$ commute.  Furthermore, the kernel of $d$ is equal to $\Pi_i \langle T_{b_i} \rangle$, so from the short exact sequence \eqref{eq:delete} we have an induced injective homomorphism
\[
\PMap(S) \hookrightarrow \PMap(d_\vB S).
\]

\subsection*{Self-doubling surfaces} With the doubling construction to hand, we now prove Theorem \ref{thm:doubling}. To this end, we are tasked with finding a surface $\bar S$ with boundary components $\vB$ such that the surfaces $S$ and $d_\vB S$ constructed above are homeomorphic.  Note that this is equivalent to the definition of self-doubling from the introduction.

\begin{proof}[Proof of Theorem \ref{thm:doubling}]
Suppose the surface $S$ is self-doubling, and let $L$ and $R$ be the subsurfaces whose interiors are homeomorphic to $S$.  We then define $\vB = \partial L$.  Following the doubling construction above, we have the injective homomorphism
\[
\PMap(S) \cong \frac{\PMap(L)}{\Pi_i \langle T_{b_i} \rangle} \hookrightarrow \PMap(S),
\]
where $b_i \in \vB$.
\end{proof}

In the remainder of this section, we give conditions on the space of ends giving rise to self-doubling surfaces.

\subsection{Ends and orbits}

A surface $S$ is \emph{stable} if for any end $e \in \Ends(S)$ there exists a sequence of nested subsurfaces $U_1 \supset U_2 \supset \dots$ defining $e$ such that $U_i \cong U_{i+1}$ \cite{FGM}.  Here, we call each $U_i$ a \emph{stable neighborhood} of $e$.

We say that an end $e \in \Ends(S)$ is of \emph{higher rank} than $e' \in \Ends(S)$ if each stable neighborhood of $e$ contains an element of the $\Map(S)$-orbit of $e'$.  Finally, we define $\vF \subset \Ends(S)$ to be the set of ends whose $\Map(S)$-orbit is finite \cite{FGM}.



We can now prove the first case of Theorem \ref{thm:doubling}.

\begin{thm}\label{thm:double-easy}
Let $S$ be a stable surface with infinitely many punctures, and suppose the genus of $S$ is either zero or infinite.  If $\vF=\emptyset$ then $S$ is self-doubling.
\end{thm}

\begin{proof}
Let $\ol S$ be the surface obtained by removing an open disk surrounding one puncture $z$.  Let $d_b S$ be the doubled surface of $\ol S$ along the sole boundary component $b$.  The induced map on the space of ends gives us the homeomorphism
\begin{align*}
( \Ends(d_b S) , \Ends_\infty(d_b S) ) &= ( \Ends(\ol S) , \Ends_\infty(\ol S) )\cup ( \Ends(\ol S), \Ends_\infty(\ol S) ) \\ 
&\cong ( \Ends(S) , \Ends_\infty(S) )\cup ( \Ends(S), \Ends_\infty(S) ).
\end{align*}
Now, suppose $e \in \Ends(S)$ has highest rank and denote the $\Map(S)$-orbit of $e$ by $O(e)$.  Since $O(e)$ is infinite, there must be an accumulation point of its elements in $\Ends(S)$.  However, as each element of $O(e)$ is of highest rank this accumulation point must also belong to $O(e)$; it follows that $O(e)$ is a Cantor set, see \cite[Proposition 4.7]{MR} for more details.

Since $S$ is stable, there exists a maximal finite set of highest rank ends $e_1, e_2, \dots, e_n$ such that $O(e_i) \neq O(e_j)$.  Indeed, if there exist infinitely many highest rank ends with different orbits then any accumulation point of such ends would not have a stable neighborhood.  Write $E(e_i)$ for the end space of a stable neighborhood of $e_i$ and $E_\infty(e_i) = E(e_i) \cap \Ends_\infty(S)$.  Since each $E(e_i)$ is a Cantor set we have that
\[
( \Ends(S) , \Ends_\infty(S) ) \cong \bigcup_i ( E(e_i), E_\infty(e_i)) \cong \bigcup_i ( E(e_i), E_\infty(e_i) ) \cup ( E(e_i), E_\infty(e_i) ).
\]
We can now conclude that
\begin{align*}
( \Ends(d_b S) , \Ends_\infty(d_b S) ) &\cong ( \Ends(S) , \Ends_\infty(S) )\cup ( \Ends(S), \Ends_\infty(S) ) \\
&\cong \bigcup_i ( E(e_i), E_\infty(e_i)) \cup ( E(e_i), E_\infty(e_i)) \\ 
&\cong ( \Ends(S) , \Ends_\infty(S) ),
\end{align*}
so $S$ and $d_\vB S$ are homeomorphic.
\end{proof}

\subsection*{Doubling along infinitely many boundary components}

We define an end to be \emph{truly non-planar} if it is non-planar and has a neighborhood that does not contain a puncture.

\begin{figure}[t]
\centering
\includegraphics[scale=1]{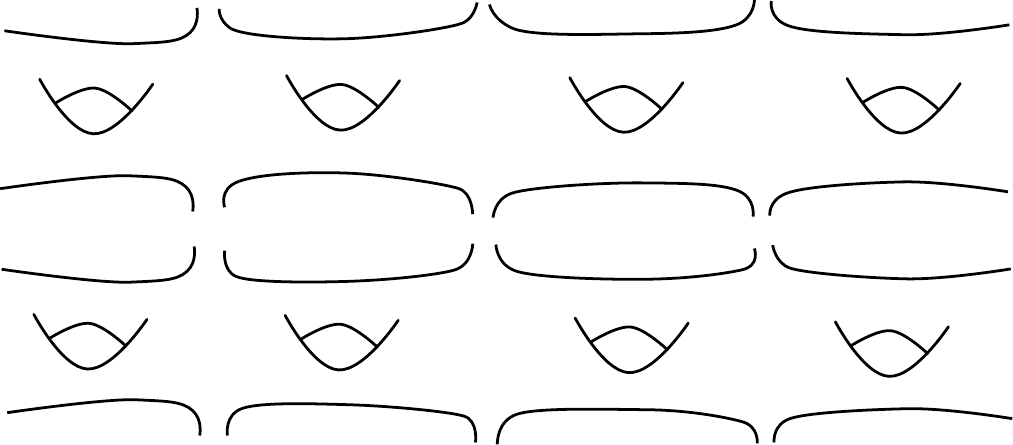}
\caption{The figure shows two copies of the surface $S$.  The double obtained using half the punctures is also homeomorphic to $S$.  There is a natural orientation reversing homeomorphism $\iota$ which fixes the boundary components of $\ol S$.}
\label{DoubleLadder}
\end{figure}

\begin{thm}\label{thm:double-hard}
Suppose $S$ is a stable surface of infinite genus with infinitely many punctures.  If $\vF$ contains no planar nor truly non-planar ends then $S$ is self-doubling.
\end{thm}

\begin{proof}
If $\vF$ is empty then we the result follows from Theorem \ref{thm:double-easy}.  As in the previous subsection, since $S$ is stable we have that $|\vF|=n$ is finite as otherwise there would be an accumulation point without a stable neighborhood \cite{FGM}.  Therefore we have a partition of $\Ends(S)$
\[
\Ends(S) = P_0 \cup P_1 \cup \dots \cup P_n
\]
such that $P_0$ contains no elements of $\vF$, and all other $P_k$ contain a single element of $\vF$.  Moreover, we may choose each $P_k$ to be the end space of a stable neighborhood of the corresponding element of $\vF$.

Our strategy is to choose an infinite collection of punctures in each set $P_k$, where $k \ge 1$, remove disk neighborhoods of these punctures producing a surface $\overline{S}$ with boundary $\vB$,  then prove that the resulting surface $d_\vB S$ is homeomorphic to $S$.  Note that since no element of $\vF$ is truly non-planar, each $P_k$ does indeed contain infinitely many punctures, when $k \ge 1$.

Suppose $k \ge 1$ and denote by $e$ the unique end in $P_k \cap \vF$.  Let $U_0 \supset U_1 \supset \dots$ be a nested sequence of homeomorphic subsurfaces with connected boundary such that $\Ends(U_0) = P_k$.  Define $X_i = U_i \sm U_{i+1}$.  Without loss of generality we may assume that $X_i \cong X_{i+1}$ and that $X_i$ contains at least two punctures (it may contain infinitely many punctures).  Now, remove an open neighborhood of one puncture in each $X_i$.  Repeating this process for each $P_k$ (where $k \ge 1$) results in a surface $\ol S$ and a set of boundary components $\vB$.  We define $d_\vB S$ as usual, and write $d_\vB (U_i) \subset d_\vB S$ for the subsurface obtained by doubling $U_i$.

Note that $d_\vB (U_0) \supset d_\vB(U_1) \supset \dots$ is a nested sequence of homeomorphic subsurfaces, in particular, it defines a unique end of $d_\vB S$.  We will now show that the end spaces of $d_\vB (U_0)$ and $U_0$ are homeomorphic.  Indeed, if each $X_i$ has infinitely many punctures, then by construction we have that
\[
\big ( \Ends( d_\vB (X_i) ), \Ends_\infty( d_\vB (X_i) ) \big  ) \cong \big ( \Ends( X_i \cup X_{i+1}), \Ends_\infty( X_i \cup X_{i+1})\big ).
\]
However, since $U_0 \supset U_2 \supset U_4 \supset \dots$ also defines the end $e$, it follows that the end spaces of $d_\vB(U_0)$ and $U_0$ are homeomorphic.  This shows that $d_\vB (U_0)$ is homeomorpic to $U_0$ with an open disk removed, that is, $d_\vB (U_0)$ has two boundary components.

If each $X_i$ contains finitely many punctures then the argument above follows similarly as both $d_\vB(U_0)$ and $U_0$ contain infinitely many punctures with a single accumulation point.  Finally, since $P_0$ contains no elements of $\vF$, the argument from Theorem \ref{thm:double-easy} implies that $P_0 \cong P_0 \cup P_0$.  It follows that

\begin{align*}
( \Ends(d_\vB S), \Ends_\infty(d_\vB S)) &\cong ( P_0 \cup P_0 ) \cup P_1 \cup \dots P_n \\ 
&\cong P_0 \cup P_1 \cup \dots P_n \\
&\cong ( \Ends(S), \Ends_\infty(S))
\end{align*}
Since $d_\vB S$ has infinite genus and no boundary components, it is homeomorphic to $S$.
\end{proof}

\section{Twist-pair preserving homomorphisms}
\label{sec:twist}
The purpose of this section is to prove Theorem \ref{thm:injections}. The key ingredient of our arguments is that a twist-pair preserving homomorphism preserves certain combinatorial configurations of curves that we term {\em spanning chains}, and which we now define; the reader should keep Figure \ref{generators} in mind: 

\begin{figure}[t]
\begin{center}
\labellist \small \hair 1pt
	\pinlabel {$\alpha_1$} at 29 105
	\pinlabel {$\alpha_2$} at 65 104
	\pinlabel {$\alpha_3$} at 105 101
	\pinlabel {$\alpha_4$} at 145 105
	\pinlabel {$\alpha_5$} at 180 100
	\pinlabel {$\alpha_6$} at 205 100
	\pinlabel {$\gamma_1$} at 225 115
	\pinlabel {$\gamma_2$} at 280 104
	\pinlabel {$\gamma_3$} at 285 74
	\pinlabel {$\gamma_4$} at 286 49
	\pinlabel {$\gamma_5$} at 283 20
	\pinlabel {$\gamma_6$} at 215 25
	\pinlabel {$\beta$} at 140 10
    \endlabellist
\includegraphics[scale=1]{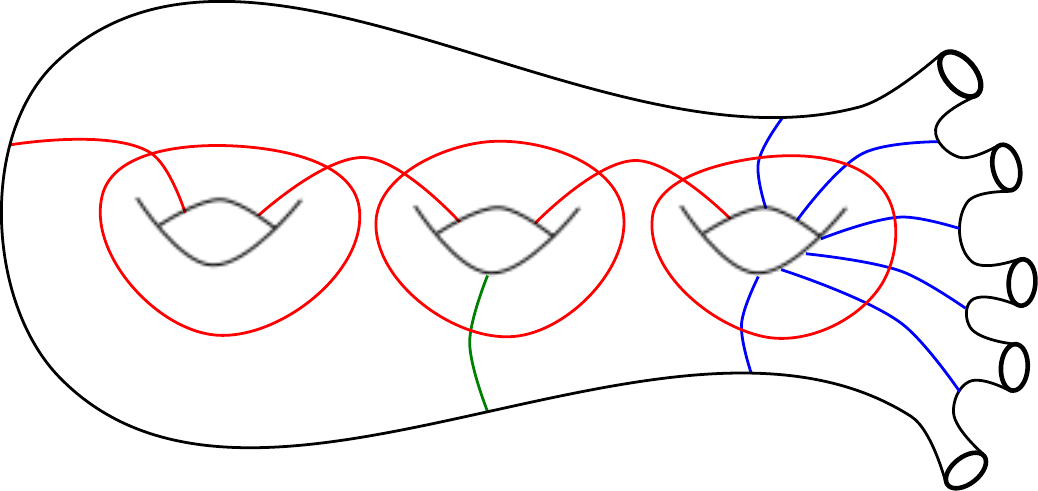}
\end{center}
\caption{A spanning chain for the surface $S_{3,5}$.}
\label{generators}
\end{figure}
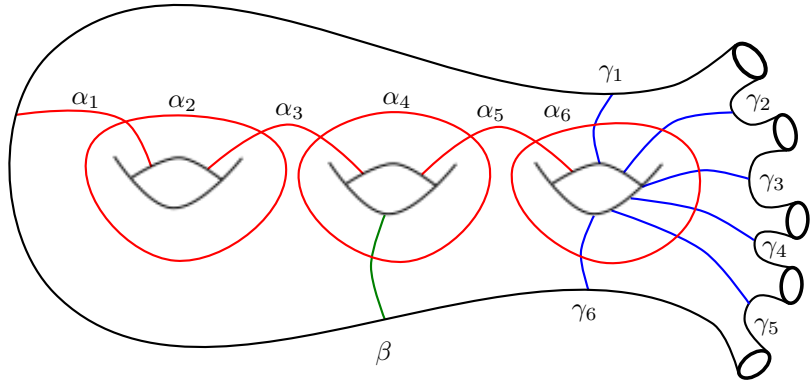

\begin{defn}
Let $S_{g,p}$ be the connected orientable surface of genus $g$ and for which the number of punctures plus boundary components is $p$.  A {\em spanning chain} is a set\[\{\beta, \alpha_1, \alpha_2, \ldots, \alpha_{2g-1}, \alpha_{2g}, \gamma_1, \ldots, \gamma_{p+1}\}\] of non-separating curves on $S_{g,p}$ such that: 
\begin{itemize}
\item $i(\beta, \alpha_4) = 1$ and $i(\beta, \alpha_j) = i(\beta, \gamma_i)= 0$ for all $i$ and all $j\ne 3$.
\item $i(\alpha_i, \alpha_{i+1}) = 1$, and $i(\alpha_i, \alpha_{j}) = 0$ otherwise. 
\item $i(\gamma_i, \alpha_{2g})= 1$, for all $i$, and $i(\gamma_i, \alpha_{j})= 0$ otherwise. 
\item $i(\gamma_i, \gamma_j)= 0$ for all $i,j$. 
\end{itemize}
\end{defn}

Observe that any surface filled by a collection of distinct curves satisfying the conditions in the definition of a spanning chain above must have genus $g$ since a regular neighborhood of the union of the curves can be reconstructed from the data, up to homeomorphism (not necessarily preserving the names of the curves); see again Figure \ref{generators}. We will need the following well-known fact; see \cite[Section 4.4.4.]{FM} for instance: 

\begin{prop}
Suppose $g\ge 1$. If $C$ is a spanning chain on $S_{g,p}$, then the Dehn twists along the elements of $C$ generate $\PMap(S_{g,p})$. 
 \label{prop:humphreys}
\end{prop} 

We will also make use of the {\em braid relation} between Dehn twists; see e. g. \cite[Section 3.5.1]{FM}:

\begin{lem}
Let $\alpha,\beta$ be curves on a surface. Then \[t_\alpha t_\beta t_\alpha = t_\beta t_\alpha t_\beta\] if and only if $i(a,b)=1$. 
\label{lem:braid}
\end{lem}

We are now ready to prove Theorem \ref{thm:injections}: 

\begin{proof}[Proof of Theorem \ref{thm:injections}]
Let $\phi: \Map(S) \to \Map(S')$ be a continuous injective twist-preserving homomorphism. In particular, $\phi$  induces a simplicial map \[\phi_*: \vC(S) \to \vC(S')\] by the rule \[\phi_*(\alpha) = \beta \iff \phi(t_\alpha) = t_\beta.\] Moreover,  $\phi_*$ is {\em superinjective}, meaning that $i(\alpha,\beta) = 0$ if and only if $i(\phi_*(\alpha), \phi_*(\beta)) \ne 0$. Finally, Lemma \ref{lem:braid} implies that $i(\phi_*(\alpha), \phi_*(\beta))= 1$ if and only if $i(\alpha,\beta)= 1$.

Now, let $Z_1 \subset Z_2 \subset \ldots$ be an exhaustion of $S$ by connected, properly embedded, $\pi_1$--injective, finite type subsurfaces for which the inclusion of $Z_i$ into $S$ induces an injection $\PMap(Z_i) \to \PMap(S)$. Without loss of generality, we may assume each component of $\partial Z_i$ is contained in the interior of $Z_{i+1}$ or is a component of $\partial S$.  Note that since each $Z_i$ is properly embedded and $\PMap(Z_i)$ injects into $\PMap(S)$, any puncture of $S$ is a puncture of $Z_i$ for some $i$.
Since the genus of $S$ is positive and $S$ has infinite type, without loss of generality we may assume that the same is true for the genus of $Z_i$ for all $i$, and that $Z_i$ is not a torus with one puncture/boundary component for any $i$.

For each $i$, choose a spanning chain $C_i$ of $Z_i$. By the discussion above,  $\phi_*(C_i)$ is a set of curves on $S'$ of the same cardinality as $C_i$, and whose elements have the same intersection pattern as the curves in $C_i$. Let $Z_i'$ be the surface obtained from the regular neighborhood of $\phi_*(C_i)$ by adding any complementary once-punctured disks or peripheral annuli which a boundary component of the neighborhood may bound.  By construction $Z_i'$ is filled by $\phi_*(C_i)$, and thus $Z_i$ and $Z_i'$ have the same genus.  Moreover, $\PMap(Z_i') < \PMap(S')$, and Proposition \ref{prop:humphreys} implies that  the homomorphism $\phi$ restricts to an injective homomorphism  \[\phi_i: \PMap(Z_i) \to \PMap(Z_i').\]   Since the genera of $Z_i$ and $Z_i'$ are equal and the homomorphism $\phi_i$ is injective, \cite[Theorem 1.1]{AS} implies that $\phi_i$ is induced by a (unique) proper $\pi_1$-injective embedding \[h_i:  Z_i \to Z_i',\] which is in fact a homeomorphism since $Z_i$ and $Z_i'$ are each filled by a spanning chain of the same cardinality and combinatorial type. 

\begin{rmk}
Although the main result of \cite{AS} is stated for surfaces of genus at least four (see the remark below the statement of Theorem 1.1 in \cite{AS}), it is also true for twist-preserving injective homomorphisms between pure mapping class groups of surfaces of the same positive genus. The reader can check that (after a suitable reduction of the target surface), the {\em standing assumption} of \cite[Section 10]{AS} holds, and that every argument goes through without modification from then on. 

We also note that the main result of \cite{AS} deals with arbitrary non-trivial homomorphisms between pure mapping class groups, and under suitable genus bounds, any such homomorphism is induced by an {\em embedding} between the underlying surfaces. In the case of an injective homomorphism, the reader will quickly verify that the definition of {\em embedding} of \cite{AS} simply means ``proper topological embedding''. 
\end{rmk}

Fix a complete hyperbolic metric with geodesic boundary on $S'$.
We view $h_i$ as a proper embedding from $Z_i$ to $S'$, and note that $h_i$ and $\phi_*$ agree on every curve in $Z_i$.   
Since $Z_i$ is more complicated than a torus with one boundary/puncture, $h_i$ is uniquely determined, up to isotopy, by this agreement with $\phi_*$ on curves.  As the punctures of $Z_i'$ were punctures of $S'$ (by construction), it follows that $h_i$ maps punctures of $Z_i$ to punctures of $S'$.
Now $Z_i \subset Z_{i+1}$, and it follows from the uniqueness statement that $h_i$ and $h_{i+1}$ agree on $Z_i$, up to isotopy.  Starting with $h_1$, and adjusting $h_2$ by isotopy if necessary, we may assume that $h_1$ and $h_2$ agree on $Z_1$.  Continuing inductively and adjusting $h_{i+1}$ to agree with $h_i$ on $Z_i$, we get a well-defined injective continuous function
\[ h \colon S \to S', \]
that agrees with $h_i$ on $Z_i$ for all $i$.  Without loss of generality, by arranging it to be the case at every step, we may assume that $h(\partial Z_i) = h_i(\partial Z_i)$ is a union of closed geodesics in the fixed hyperbolic metric.
Since any curve in $S$ is contained in $Z_i$ for some $i$, it follows that $h(\delta) = \phi_*(\delta)$ for every curve $\delta$ on $S$. 

We now claim that $h$ is proper. To see this, suppose for contradiction that this were not the case. Then there exists a compact set $K$ of $S'$ such that $h^{-1}(K)$ is not compact. We may enlarge $K$ if necessary to a finite type, connected, properly embedded subsurface of $W \subset S'$ with geodesic boundary (in our fixed hyperbolic metric).
Then $h^{-1}(W)$ is a non-empty, non-compact, closed subset of $S$.  Now we observe that $h^{-1}(W) \cap \partial Z_i \neq \emptyset$ for all $i$ sufficiently large: otherwise for some $i$, $W$ would be entirely contained in $h(Z_i) = h_i(Z_i)$, and hence so would $K$, implying that $h^{-1}(K) = h_i^{-1}(K)$ is compact (by properness of $h_i$), a contradiction.
Therefore, we can find components $\alpha_i \subset \partial Z_i$ so that for all $i$ sufficiently large, $h(\alpha_i)$ transversely and essentially intersects $W$.
 In particular, the sequence $\{t_{\alpha_i}\}$ of mapping classes converges to the identity in $\Map(S)$. In contrast, the image sequence $h(\alpha_i)$ is a sequence of pairwise distinct curves, all of which intersect $W$. As a consequence, the sequence $\{t_{h(\alpha_i)}\}$ cannot converge to the identity in $\Map(S')$. This contradicts the fact that $\phi$ is continuous, as \[\phi(t_{\alpha_i}) = t_{\phi_*(\alpha_i)}= t_{h(\alpha_i)}.\]

Since $h$ is a proper embedding it induces $h_\sharp \colon \PMap(S) \to \PMap(S')$, an injective homomorphism.

By hypothesis, $S$ either has empty boundary or else has at most one end accumulated by genus. In the former case, it follows that $h$ is in fact a homeomorphism, and thus $h_\sharp$ and $\phi$ are equal by Theorem~\ref{T:Alexander}. In the latter case, we know that  $h$ and $\phi_*$ agree on every isotopy class of simple closed curve, and hence $h_\sharp$ and $\phi$ agree on every Dehn twist.  Since Dehn twists topologically generate $\PMap(S)$ \cite{PV} and $\phi$ is continuous, it follows that $\phi$ and $h_\sharp$ are equal, as required.
\end{proof}

\begin{rmk} (1) As mentioned in the introduction, Theorem \ref{thm:injections} no longer holds if $\partial S\ne \emptyset$ and $S$ has more than one end accumulated by genus. Indeed, suppose that $S$ has at least two ends accumulated by genus, in which case there exists a surjective homomorphism  $\rho: \PMap(S) \to \ZZ$  by \cite{APV}. Since $\partial S \ne \emptyset$, we may find a surface $S'$  for which there exists a $\pi_1$-injective embedding $\iota: S \to S'$ and such that  the mapping class group of $S' \setminus S$ is infinite.  

Consider the injective homomorphism $\phi: \PMap(S) \to\PMap(S')$ induced by this embedding, noting that its image is supported on $\PMap(\iota(S))$. Now choose an infinite-order element $g \in \PMap(S')$ supported on $S' \setminus \iota(S)$. Using the homomorphism $\rho$ we  construct a homomorphism $\rho_g: \PMap(S) \to \PMap(S')$ whose image is the cyclic group generated by $g$; in particular, the image of $\rho_g$ is supported on $\PMap(S' \setminus \iota(S))$

As the images of $\phi$ and $\rho_g$ commute, we may consider the homomorphism $(\phi, \rho_g): \PMap(S) \to \PMap(S')$,  whose image is contained in \[\PMap(\iota(S)) \times \PMap(S' \setminus \iota(S))< \PMap(S').\] Observe that this homomorphism is continuous, injective (since $\phi$ is) and twist-preserving (since the $\rho$-image of every Dehn twist   is trivial, by \cite{APV}). However, it  is not induced by any subsurface embedding $S \to S'$, as required.

\medskip 

\noindent (2) As may have become apparent in the proof of Theorem \ref{thm:injections}, the requirement that the homomorphism $\phi$ be injective may be relaxed in various ways. For instance, say that a twist-preserving homomorphism {\em preserves twist pairs} if two Dehn twists commute if and only if their images commute. With this definition, a minor modification of the  proof above yields: 

\medskip

\noindent {\em Let $S$ and $S'$ be  surfaces of infinite type, and assume $S$ has positive genus. Assume further that either the boundary of $S$ is empty, or else $S$ has at most one end accummulated by genus.   If $\phi: \PMap(S) \to \PMap(S')$ a continuous twist-pair preserving homomorphism, then there is a proper $\pi_1$-injective embedding $h: S \to S'$ that induces $\phi$. }

\medskip

Indeed, the only subtlety when mimicking the proof above occurs in justifying why the restriction homomorphism $\phi_i:\PMap(Z_i) \to \PMap(Z_i')$ is injective. To this end, if a non-central element $f\in \PMap(Z_i)$ is in the kernel of $\phi_i$, we may use its Nielsen-Thurston normal form to find a curve $\alpha \subset Z_i$ such that $i(f(\alpha),\alpha) >0$. However, since $f \in \ker(\phi_i)$ we have that $i(\phi_*(f(\alpha)), \phi_*(\alpha))= 0$, which contradicts that $\phi_*$ is superinjective. 

If on the other hand, there is $f\in \ker(\phi_i)$ which is central, then we have an induced injective homomorphism between the pure mapping class groups modulo (the relevant part of) their centers, at which point the main result of \cite{AS} applies. 
\end{rmk}

\section{Superinjective maps of curve graphs}\label{sec:graphs}
Finally, in this section we give a proof of Theorem \ref{thm:superinjective}. We separate the proof into two parts.

\begin{proof}[Proof of part (1) of Theorem \ref{thm:superinjective}]
Equip $S$ with a fixed hyperbolic metric, and realize every vertex of $\vC(S)$ by its unique geodesic representative in its homotopy class. 

Suppose first that $S$ has infinite genus. By choosing two points $p,q$ in the complement of the union of all simple closed geodesics on $S$, we obtain a (non-surjective) superinjective map $\vC(S) \to \vC(S\smallsetminus \{p,q\})$. Now, observe that $\vC(S\smallsetminus \{p,q\}) \cong \vC(S')$, where $S'$ is obtained by removing two open discs (with disjoint closures) from $S$. Finally, we may glue a cylinder to $S'$ along its boundary components, obtaining a surface $S'' \cong S$ and the inclusion $S' \to S''$ induces a superinjective map $\vC(S') \to \vC(S'') \cong \vC(S)$.  The composition of these superinjective maps/isomorphisms
\[ \vC(S) \to \vC(S \smallsetminus \{p,q\}) \cong \vC(S') \to \vC(S'') \cong \vC(S) \]
is a superinjective map which is not surjective (since the first map is non-surjective).

Suppose now that $S$ has finite genus.  Then either $S$ has infinitely many punctures or its space of ends is a Cantor set union a finite set.  In the former case we may puncture $S$ at a point $p$ in the complement of the union of simple closed geodesics, producing a non-surjective, superinjective map
\[ \vC(S) \to \vC(S \smallsetminus \{p\}) \cong \vC(S), \]
where the isomorphism comes from fact that $S$ and $S \smallsetminus \{p\}$ are homeomorphic.  In the latter case, we may again remove a point $p$ missing all simple closed geodesics to get a non-surjective, superinjective map $\vC(S) \to \vC(S \smallsetminus \{p\}) \cong \vC(S')$, where $S'$ is obtained from $S$ by removing an open disk, then glue in a disk minus a Cantor set to $S'$ produce a surface $S'' \cong S$.  This gives a superinjective map $\vC(S') \to \vC(S'') \cong \vC(S)$, and composing with the maps above gives the required non-surjective, superinjective map in this case.
\end{proof}

We are now going to prove part (2) of Theorem \ref{thm:superinjective}. The proof is divided into two lemmas. In what follows, we denote  by $M$ the Loch Ness Monster surface, which we again recall is the unique, up to homeomorphism, infinite-genus surface with exactly one end. 

\begin{lem}
Let $S$ be an arbitrary infinite-type surface. Then there exists a superinjective map $\vC(S) \to \vC(M)$. 
\label{lem:intomonster}
\end{lem}

\begin{proof}
We construct the desired map via a series of intermediate steps. First, let $S_1$ be the surface obtained from $S$ by first replacing each puncture by a boundary component, and then gluing, to each of these new boundary components, a torus with one boundary component. Observe that if $S_1$ has any planar ends, then they must belong to a Cantor set.  We have a superinjective map 
\begin{equation}
\vC(S) \to \vC(S_1).
\label{eq:1}
\end{equation}
Fix a principal exhaustion $P_1 \subset P_2 \subset \dots$ of $S_1$, namely an exhaustion of $S_1$ by compact subsurfaces such that every component of $\partial P_i$ is separating in $S_1$.  Let $X^n_1, X^n_2, \dots, X^n_{k_n}$ be the connected components of $P_n \sm P_{n-1}$. We puncture $S_1$ along $y^n_i \in X^n_i$, and then replace each of these punctures by a boundary component $b^n_i$, obtaining a new surface $S_2$ and a superinjective map
\begin{equation}
\vC(S_1) \to \vC(S_2)
\label{eq:2}
\end{equation}
The surface $S_2$ is naturally equipped with a principal exhaustion $Q_1 \subset Q_2 \subset \dots$ coming from the obvious subsurface embedding $S_2 \to S_1$; abusing notation, we  denote the connected components of $P_n \sm P_{n-1}$ in $S_2$ by $Z^n_1, Z^n_2, \dots, Z^n_{k_n}$.

 Now, we glue a sphere with $k_n$ boundary components to the union of the $Z_i^n$, thus obtaining a connected surface $Y_n$  which contains $Q_n$, see Figure \ref{Monster}. Moreover, $Y_n \subset Y_{n+1}$, so we may consider 
$S_3 = \bigcup Y_n$.
Since $S_2 = \bigcup Q_n$, we have that $S_3$ contains a subsurface homeomorphic to $S_2$, and in particular there is a superinjective map 
\begin{equation}
\vC(S_2) \to \vC(S_3)
\label{eq:3}
\end{equation}
It remains to show that $S_3$ is homeomorphic to $M$. By construction, $S_3$ has infinite genus and no planar ends.  Since the complement of any finite-type subsurface of $S_3$ contains only one component of infinite-type, we have that $S$ has a single end which is not planar. Therefore, $S_3$ is homeomorphic to $M$, as desired. At this point, the composition of the maps in \eqref{eq:1},\eqref{eq:2} and \eqref{eq:3} gives the superinjective map \[\vC(S) \to \vC(M).\] This finishes the proof of the lemma. 
\end{proof}

\begin{figure}[t]
\centering
\includegraphics[scale=0.3]{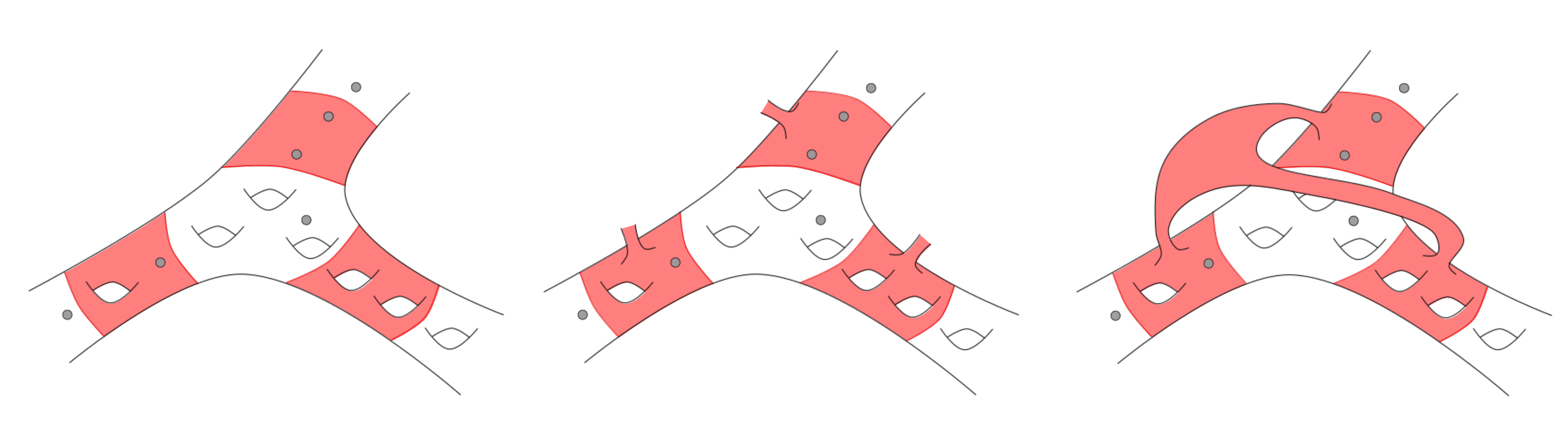}
\caption{One of the steps towards obtaining the surface $S_3$}
\label{Monster}
\end{figure}

Next, we prove that every surface $S'$ of infinite genus contains a subsurface homeomorphic to $M$.  

\begin{lem}\label{lem:monster}
If $S$ is a surface with infinite genus then it contains a subsurface $S \cong M$. In particular, there is a superinjective map $\vC(M) \to \vC(S)$.
\end{lem}

\begin{proof}
Let $N$ be a neighborhood of a non-planar end $e$.  Let $c_1, c_2, \dots \subset N$ be a set of pairwise disjoint separating curves, each of which bound bounding a one-holed torus.  Let $a_i$ be an arc connecting $c_{i-1}$ and $c_i$ such that $a_i \cap a_j = \emptyset$.

Now, the boundary of a neighborhood of $\bigcup a_i \cup c_i$ contains a separating arc $\alpha$.  We define $\Sigma$ to be the component of $S \sm \alpha$ containing every $c_i$ (here, we take the complement $S \sm \alpha$ to be without boundary).  It follows that $\Sigma$ has infinite genus.  Moreover, the complement of any compact subsurface contains a single non-compact component, so $\Sigma$ has a single end and is therefore homeomorphic to $M$.
\end{proof}

Finally, we put together the above pieces in order to prove Theorem \ref{thm:superinjective}: 

\begin{proof}[Proof of Theorem \ref{thm:superinjective}]
Let $S$ and $S'$ be surfaces of infinite type, and assume $S'$ has infinite genus. By Lemma \ref{lem:intomonster}, there is a superinjective map \[\vC(S) \to \vC(M).\] In turn, Lemma \ref{lem:monster} tells us there exists a superinjective map \[\vC(M) \to \vC(S'),\] from which the result follows. 
\end{proof}

\end{document}